\documentclass[a4paper,legno,11pt]{article}

\usepackage[utf8]{inputenc}
\usepackage[T1]{fontenc}
\usepackage{xspace,
	amsthm,
	amsmath,
	amssymb,
	bbm,
	tikz,
	graphicx,
	keyval,
       xcolor,
  pgfplots}
\usepackage{float}
\usepackage{subfig}

\DeclareGraphicsExtensions{.png, .pdf}
\DeclareMathOperator\arctanh{arctanh}

\usepackage[hidelinks]{hyperref}

\usepackage{wrapfig}
\usepackage{lipsum}
\usepackage{caption}

\usepackage{multicol} 

\usepackage{geometry}
\tikzstyle{nodino}=[circle,draw,fill,inner sep=0pt,minimum size=0.5mm]
\tikzstyle{infinito}=[circle,inner sep=0pt,minimum size=0mm]
\tikzstyle{nodo}=[circle,draw,fill,inner sep=0pt, minimum size=0.5*width("k")]
\tikzstyle{nodo_vuoto}=[circle,draw,inner sep=0pt, minimum size=0.5*width("k")]
\usetikzlibrary{graphs}
\tikzset{every loop/.style={min distance=10mm,in=300,out=240,looseness=10}}
\tikzset{place/.style={circle,thick,draw=blue!75,fill=blue!20,minimum
		size=6mm}}
\tikzset{place2/.style={circle,thick,draw=red!75,fill=red!20,minimum
		size=6mm}}

\usepackage{xcolor}

\definecolor{DEblue}{rgb}{0.08,0.24,0.54}
\definecolor{DEgreen}{rgb}{0.07,0.76,0.75}
\definecolor{DEorange}{rgb}{0.92,0.50,0.04}		

\newcommand{\rr}{{\mathbb R}}

\newcommand{\dip}{\eta^{dip}}

\newcommand{\Htau}{H_\tau^1}

\newcommand{\om}{\omega}

\newcommand{\dx}{\,dx}
\newcommand{\dt}{\,dt}
\newcommand{\ds}{\,ds}

\newcommand{\intRneg}{\int_{-\infty}^{0}}
\newcommand{\intRpos}{\int_{0}^{+\infty}}

\newcommand{\intone}{\int_{-1}^{1}}

\theoremstyle{plain} 
\newtheorem{thm}{Theorem}[section] 
\newtheorem{cor}[thm]{Corollary} 
\newtheorem{lem}[thm]{Lemma} 
\newtheorem{prop}[thm]{Proposition} 

\theoremstyle{definition}

\theoremstyle{definition} 
\newtheorem{defn}{Definition}[section]

\theoremstyle{remark} 
\newtheorem{rem}{Remark}[section]

\title{Discontinuous Ground States for NLSE on $\rr$ with  a F\"{u}l\"{o}p-Tsutsui $\delta$ interaction}

\author{Riccardo Adami$^1$, Takaaki Nakamura$^{2}$, Alice Ruighi$^{1,3}$
\\ \ \\{\small  {$^1$}Dipartimento di Scienze
Matematiche ``G.L. Lagrange'', Politecnico di Torino } \\ {\small
Corso Duca degli Abruzzi, 24, 10129 Torino, Italy} \\ \ \\
{\small  {$^2$}Laboratory of Physics, Kochi University of Technology} \\ {\small
Tosa Yamada, Kochi 782-8502, Japan}\\ \ \\
{\small  {$^3$}Dipartimento di Matematica ``G. Peano'',  Universit\`a di Torino } \\ {\small
Via Carlo Alberto 10, 10123 Torino, Italy}
}

\date{}
\begin{document}
\maketitle

\abstract
	We analyse the existence and the stability of the ground states of the one-dimensional nonlinear Schr\"{o}dinger equation with a focusing power nonlinearity and a defect located at the origin. In this paper a ground state is defined as a global minimizer of the action functional on the Nehari manifold and the defect considered is a F\"{u}l\"{o}p-Tsutsui $\delta$ type, namely a $\delta$ condition that allows discontinuities. The existence of ground states is proved by variational techniques, while the stability results from the Grillakis-Shatah-Strauss theory. 

       \section{Introduction}
       \label{sec:intro}
       
In the last decades the study of dynamics on metric graphs has developed enormously. One of the main reasons of this interest can be found in the adaptability of the models in approximating the evolution of systems located on ramified structures, where the transverse dimensions are negligible if compared to the longitudinal ones.\\
Schr\"odinger dynamics on ramified structures, or networks, was first investigated by Ruedenberg and Scherr \cite{rued} in 1953, studying the energy spectrum of valence electrons on the array of the naphthalene molecules. In particular, exploiting the geometry of the molecular structure of naphthalene, they defined a suitable Schr\"{o}dinger operator on the edges of a hexagonal grid in order to represent the quantum energy of the system and then computed its spectrum. 
This seminal paper has not only been considered as a milestone in physical chemistry, but it also introduced some important mathematical tools, such as the Kirchhoff's conditions at the vertices of a ramified structure. In some sense, this type of conditions describes a situation of homogeneity in the medium in which the dynamics occurs and they have been deeply investigated, at first for linear dynamics (see for example the milestone paper by Kostrykin and Schrader \cite{kost} or the introductory book by Berkolaiko and Kuchment \cite{berk}) and then for the nonlinear ones. In particular, the forerunner of this last line of research is the treatise \cite{meh}, but only in the last few years this topic has been deeply investigated \cite{ast1,ast2,ast3}, focusing on particular types of metric graphs such as compact graphs \cite{bmp,simon,good,cds,marpel}, periodic graphs \cite{adst,adr,dov} and the infinite metric trees \cite{dst}. Finally, the problem of a nonlinearity concentrated on a subgraph has been variously explored, for instance in \cite{t,serraten,serraten2,dt}.\\
Next to these conditions, there exists the family of non-Kirchhoff's conditions that, on the other hand, represents an inhomogeneity or defect in the medium. Examples of these conditions are the $\delta$ conditions, explored both on the real line and on star graphs \cite{fj,foo,acfn2012,acfn2013,acfn2016,anv}, the $\delta'$ conditions \cite{anv,an} and the dipole conditions \cite{anv}. More recently, a new type of non-Kirchhoff's conditions have been studied, the so-called nonlinear delta, for which two nonlinearities coexist: the standard one and a pointwise one \cite{abd,bd}. For a more complete investigation on the non-Kirchhoff's conditions we refer to \cite{abr}. \\
The purpose of this paper is to present some results on the study of the Nonlinear Schr\"{o}dinger equation when a specific generalization of $\delta$ conditions are imposed at the origin of the real line. Following \cite{cheon} we shall call them \emph{F\"{u}l\"{o}p-Tsutsui $\delta$} conditions and roughly speaking they can be seen as $\delta$-type conditions that generate discontinuities where the defect is located.
\newline 
\newline To be more specific, we investigate the existence and the stability of ground states on the real line $\rr$ for the Nonlinear Schr\"{o}dinger equation 
\begin{equation}
\label{eq:schroed}
i \partial_t u = H_{\tau,v} u - |u|^{ 2\mu}u,
\end{equation}
\noindent where $H_{\tau,v}$ is the self-adjoint extension of the one-dimensional laplacian, defined on the domain
\begin{equation}
\label{eq:domain}
D(H_{\tau,v}) := \{ u \in H^2(\mathbb{R} \backslash \{0\}) \, : \, u(0+)=\tau u(0-), u'(0-)-\tau u'(0+)=v u(0-)\}
\end{equation}
\noindent and its action reads $(H_{\tau,v} u)(x)= -u''(x)$ out of the origin. \\ In \eqref{eq:domain}, $\tau \in \rr \backslash \{0,\pm 1\}$ and  $v>0$, namely, we consider the case of an attractive $\delta$ interaction only.\\
In \cite{adno}, it has been established that the energy space associated to equation \eqref{eq:schroed} is 
\begin{equation*}
\Htau : =\{u \in H^1(\mathbb{R}_-) \oplus H^1(\mathbb{R}_+) \, : \, u(0+)=\tau u(0-)\},
\end{equation*}
and the energy functional 
\begin{equation*}
\label{eq:energy}
E(u)=\frac{1}{2} \left( ||u'||^2_{L^2(\mathbb{R}_-)}+||u'||^2_{L^2(\mathbb{R}_+)} \right) -\frac{1}{2\mu+2} ||u||_{L^{2\mu+2}(\mathbb{R})}^{2\mu+2} - \frac{v}{2}|u(0-)|^2
\end{equation*}
is conserved by the flow defined by \eqref{eq:schroed}.\\\\
In the following we use the slight abuse of notation: $$||u'||^2_{L^2(\mathbb{R})}=||u'||^2_{L^2(\mathbb{R}_-)}+||u'||^2_{L^2(\mathbb{R}_+)}$$ and, if it is not confusing, we shorten $||u'||^2_{L^2(\mathbb{R})}$ with $||u'||^2_2$ and $||u||^p_{L^p(\mathbb{R})}$ with $||u||^p_p$ for any exponent $p \geq 2$.\\\\
We define a \emph{ground state} as a global minimizer of the action functional
\begin{equation*}
\label{eq:action}
S_\om(u) = E(u) +\frac{\om}{2} ||u||^2_2,
\end{equation*}
\noindent among all functions in $\Htau$ satisfying the Nehari's constraint $I_\om(u) = 0$, where
\begin{equation*}
\label{eq:nehari}
I_\om(u) = ||u'||_2^2 - ||u||_{2\mu+2}^{2\mu+2} - v|u(0-)|^2+\om ||u||_2^2
\end{equation*}
\noindent is called Nehari's functional and $\mu >0$.\\
Notice that the notion of ground state we shall use does not refer to the mass constraint, so that its orbital stability is not guaranteed a priori. On the other hand, the use of Nehari manifold in the study of ground states is classical \cite{shatah} and has been already introduced for the study of Schr\"odinger equation with point interactions in \cite{fj}, \cite{foo} and \cite{an}.\\ Following the line of these works, we find stationary states and compare them to establish which, among them, are the ground states. This makes our model richer than the one described in \cite{fj} and \cite{foo}, encompassing a pure $\delta$ interaction. For this feature, the present model can be considered as a bridge between $\delta$ and $\delta'$ models.\\\\
The paper is organized as follows: in Section 2 we collect some preliminary results relevant for the remaining part of the paper; the main theorem about the existence of the ground states will be presented in Section 3, whereas in Section 4 we study the stationary states of the constrained functional and identify the ground state among them; Section 5 is finally devoted to the study of the orbital stability of the ground states.


	\section{Basic facts}
	\label{sec:notation}
	
In this section we will collect some basic remarks and preliminary results that will be relevant in the following, but for simplicity and clarity we prefer to present them here. \\As outlined in the Introduction, since one of the subjects of our study will be the existence of non-vanishing global minimizers for the action functional under the Nehari's constraint, let us recall that the stationary states of the functional $S_\om$ belong to the Nehari manifold, namely the zero-level set of the Nehari's functional and that is the reason why people refer to the Nehari's constraint as the "natural constraint" for the action functional.\\
To our aim, let us define a further functional, called reduced action, that does not depend on $\om$
\begin{equation*}
\widetilde{S}(u):= \frac{\mu}{2(\mu+1)}||u||^{2\mu+2}_{2\mu+2}
\end{equation*}
and note that $S_\om(u)=\widetilde{S}(u)$ holds for every $u$ on the Nehari manifold.\\
The importance of this functional is clarified by the following Lemma \ref{lemma2}. 
\begin{rem} Let us note that the energy of the linear bound states in $\Htau$ is $\om = \frac{v^2}{(\tau ^2 + 1)^2}$. Indeed, if we consider the eigenvalue problem
\begin{equation}
\label{linear}
\Bigg \{ \begin{array}{lr}
-u''+\omega u =0, \quad x\neq0, \quad u\in H^2(\mathbb{R}\backslash \{0\})\\
u(0+) =\tau u(0-), \\
u'(0-)-\tau u'(0+) =v u(0-),  \\
\end{array}
\end{equation}
we know that $u(x)=\chi_- e^{\sqrt{\om}x}+ \chi_+ e^{-\sqrt{\om}x}$, where $\chi_\pm$ are the characteristic functions of $\rr_\pm$, solves the first equation in \eqref{linear}. Imposing the boundary conditions at the origin on such $u$, it follows $$e^{-\sqrt{\om}(0+)}= \tau e^{\sqrt{\om}(0-)}$$ and $$\sqrt{\om}e^{\sqrt{\om}(0-)}+\tau \sqrt{\om} e^{-\sqrt{\om}(0+)}=v e^{\sqrt{\om}(0-)}.$$\\
Hence, $\om = \frac{v^2}{(\tau ^2 + 1)^2}$.
\end{rem}
\begin{lem}
\label{lemma2}
Let $\om > \frac{v^2}{(\tau ^2 + 1)^2}$. Then 
\begin{align} 
\label{inf1}
d(\om) : & = \inf \{S_\om(u) \, : \, u \in \Htau \backslash \{ 0 \}, I_\om(u)=0 \} \\
\label{inf2}
         & = \inf \{\widetilde{S}(u) \, : \, u \in \Htau \backslash \{ 0 \}, I_\om(u) \leq 0 \}.
\end{align}
In particular, if $u$ is a minimizer for one problem, it is a minimizer also for the other.
\begin{proof}
We can split the proof in two steps. In the first one we will show the equivalence between (\ref{inf1}) and (\ref{inf2}), whereas in the second one the equivalence between the two minimizers will be proved.\\\\
\emph{Step 1:} let $u \in \Htau \backslash \{ 0 \} $ such that $I_\om(u)=0$. Then $S_\om(u)=\widetilde{S}(u)$ and 
\begin{equation*}
\inf \{ S_\om(u)\, : \, I_\om(u) =0 \} \quad \geq \quad \inf \{ \widetilde{S}(u) \, : \,  I_\om(u) \leq 0 \}.
\end{equation*}
On the other hand, if we choose $u \in \Htau \backslash \{ 0 \} $ such that $I_\om(u)<0$, 
we can define 
\begin{equation}
\label{alfa}
\alpha(u):=\left( \frac{||u'||^2_2-v|u(0-)|^2+\om ||u||^2_2}{||u||^{2\mu+2}_{2\mu+2}} \right)^ {\frac{1}{2\mu}}.
\end{equation}
\noindent Because of the hypothesis  $I_\om(u)<0$, it follows that $\alpha(u)<1$. Moreover $I_\om(\alpha(u)u)=0$, hence $S_\om(\alpha(u)u)=\widetilde{S}(\alpha(u)u)=\alpha(u)^{2\mu+2}\widetilde{S}(u)<\widetilde{S}(u)$ and
\begin{equation*}
\inf \{ S_\om(u)\, : \, I_\om(u) =0 \} \quad \leq \quad \inf \{ \widetilde{S}(u) \, : \,  I_\om(u) \leq 0 \}.
\end{equation*} 
Hence, \eqref{inf1} and \eqref{inf2} are equivalent.\\
\emph{Step 2:} if $u$ is a minimizer for the functional $S_\om$ and $I_\om(u)=0$, then it means that there exists a function that reaches the infimum also for the problem with the functional $\widetilde{S}$. On the other hand, if $u$ were a minimizer for $\widetilde{S}$ with $I_\om(u)<0$, we could define $\alpha(u)$ as before and again it would result that $\widetilde{S}(\alpha(u)u)<\widetilde{S}(u)$. But this would contradict the fact that $u$ is a minimizer, hence $I_\om(u)=0$ and $u$ turns out to be a minimizer also for $S_\om$.
\end{proof}
\end{lem}
\begin{rem}
\label{idea}
In the following we use that $I_\om(u)<0$ cannot hold if $u$ is a minimizer. 
\end{rem}
\noindent \\We now present a Sobolev type inequality adapted to the space $\Htau$, endowed with the norm 
\begin{equation}
\label{Htaunorm}
||u||_{\Htau}^2 :=||u||_{L^2(\rr)}^2+||u'||_{L^2(\rr_-)}^2+||u'||_{L^2(\rr_+)}^2
\end{equation}

\begin{prop}[Sobolev inequality]
\label{prop:Sobolev}
For any $u \in \Htau$, 
\begin{equation}
\label{Sobolev}
||u||_{2\mu+2} \leq C ||u||_{\Htau}
\end{equation}
where $C$ is a positive constant which depends only on $\mu$.
\begin{proof}
Let us consider $u \in \Htau$ such that $u=\chi_-u_-+\chi_+ u_+$ where $u_\pm$ are even functions in $H^1(\mathbb{R})$ and $\chi_\pm$ are the characteristic functions of $\rr_\pm$. 
\begin{align*}
||u||^2_{2\mu+2} & = \left( ||u||^{2\mu+2}_{2\mu+2} \right) ^{\frac{2}{2\mu+2}} =  \left( \frac{1}{2} \left( ||u_+||^{2\mu+2}_{2\mu+2} + ||u_-||^{2\mu+2}_{2\mu+2} \right) \right) ^{\frac{2}{2\mu+2}} \\
                            & \leq \frac{1}{2^{\frac{2}{2\mu+2}}}\left( ||u_+||^{2}_{2\mu+2} + ||u_-||^{2}_{2\mu+2}\right)\\
                            & \leq C \left( ||u_+||^{2}_{H^1} + ||u_-||^{2}_{H^1}\right)\\
                            & = C \left( ||u_+||^{2}_{2} + ||u'_+||^{2}_{2} + ||u_-||^{2}_{2} + ||u'_-||^{2}_{2}\right)\\
                            & = C \left(||u||^{2}_{2} + ||u'||^{2}_{2}\right) = C ||u||^2_{\Htau}.
\end{align*}
where the inequalities follow noting that $\frac{2}{2\mu+2}<1$ and by the Sobolev embedding on the line.

\end{proof}
\end{prop}

  \section{Existence}
  \label{sec:existence}	
In this section we present the main result concerning the existence of ground states, i.e. minimizers for the action functional under the Nehari's constraint. More precisely we prove the following theorem. 

\begin{thm}
\label{existence}
Let $\om > \frac{v^2}{(\tau ^2 + 1)^2}$. Then there exists $u \in \Htau \backslash \{ 0 \}$ that minimizes $S_\om$ among all functions belonging to the Nehari manifold $I_\om(u)=0$.
\end{thm}

\noindent The proof follows the line of \cite{an} and exploits Banach-Alaoglu's theorem and Brezis-Lieb's lemma to obtain convergence of minimizing sequences. However, before proving Theorem \ref{existence} we present some preliminary lemmas that show that the functional $S_\om$ is bounded from below and this motivates our search for the ground states.

\begin{lem}
\label{lemma1}
For any  $\om > \frac{v^2}{(\tau ^2 + 1)^2}$, it holds 

\begin{equation}
\label{cane}
||u'||^2_2-v|u(0-)|^2+\om||u||^2_2 \geq C ||u||^2_{\Htau},
\end{equation}
for some constant $C>0$.
\begin{proof}
First of all let us consider $u \in \Htau$ such that $u=\chi_- u_- + \chi_+ u_+$ where $u_\pm$ are even functions in $H^1(\mathbb{R})$ and $\chi_\pm$ are the characteristic functions of $\rr_\pm$. Note that, thanks to the standard Gagliardo-Nirenberg inequality in $H^1(\mathbb{R})$ and by symmetry, it follows that
\begin{align*}
|u(0 \pm)|^2 \leq ||u_\pm||^2_\infty & \leq ||u_\pm||_2 ||u'_\pm||_2  \\
                                                      & = \left( \int_{-\infty}^{+\infty} |u_\pm|^2 \right)^{\frac{1}{2}} \left( \int_{-\infty}^{+\infty} |u'_\pm|^2 \right)^{\frac{1}{2}}\\
                                                      & = \left( 2 \int_{-\infty}^{+\infty} |\chi_\pm u_\pm|^2 \right)^{\frac{1}{2}} \left(2 \int_{-\infty}^{+\infty} |\chi_\pm u'_\pm|^2 \right)^{\frac{1}{2}}\\
                                                      & = 2 ||\chi_\pm u_\pm||_2 ||\chi_\pm u'_\pm||_2.
\end{align*}
Let us observe that, in order to get (\ref{cane}), it is sufficient to estimate the negative term in the inequality. In particular, thanks to the fact that $u(0+)=\tau u(0-)$, for any $\alpha \geq 0$
\begin{equation}
\label{punt}
v|u(0-)|^2  = v \alpha |u(0-)|^2 +\frac{v(1-\alpha)}{\tau ^2}|u(0+)|^2.
\end{equation}
Hence, using the previous estimate on the r.h.s. of \eqref{punt}, we obtain
\begin{equation*}
\label{aa}
v|u(0-)|^2 \leq 2v \alpha  ||\chi_- u_-||_2 ||\chi_- u'_-||_2 + 2\frac{v (1-\alpha)}{\tau^2}  ||\chi_+ u_+||_2 ||\chi_+ u'_+||_2.
\end{equation*} 
Choosing $\alpha = \frac{1}{\tau^2+1}$ we get
\begin{equation*}
v|u(0-)|^2 \leq \frac{2v}{\tau^2+1} \left( ||\chi_- u_-||_2 ||\chi_- u'_-||_2 +  ||\chi_+ u_+||_2 ||\chi_+ u'_+||_2 \right)
\end{equation*}
and, for all $a>0$ 
\begin{align*}
v|u(0-)|^2 & \leq \frac{2v}{\tau^2+1} \left( \frac{a}{2}||\chi_- u_-||^2_2  + \frac{1}{2a}||\chi_- u'_-||^2_2 +  \frac{a}{2}||\chi_+ u_+||^2_2 + \frac{1}{2a}||\chi_+ u'_+||^2_2 \right)\\
                    & = \frac{v}{\tau^2+1}\left( a||u||^2_2 + \frac{1}{a} ||u'||^2_2\right).
\end{align*}
Finally, we obtain
\begin{align*}
||u'||^2_2-v|u(0-)|^2+\om||u||^2_2 & \geq \left( 1- \frac{v}{a(\tau^2 +1)} \right) ||u'||^2_2 + \left( \om - \frac{va}{\tau^2 +1} \right) ||u||^2_2 \\
                                                              &  \geq C ||u||^2_{\Htau},
\end{align*}
where the constant $C$ is positive since we can always choose a parameter $a$ such that 
\begin{equation}
\label{eq:ebounds}
\frac{v}{\tau^2+1}<a<\frac{\om (\tau^2+1)}{v},
\end{equation} 
thanks to the hypothesis $\om > \frac{v^2}{(\tau ^2 + 1)^2}$.

\end{proof}
\end{lem}

\noindent In particular, it follows:

\begin{lem}
\label{below}
For any $\om > \frac{v^2}{(\tau ^2 + 1)^2}$, it holds $d(\om)>0$.
\begin{proof}
This result is a consequence of Lemma \ref{lemma1}, since
\begin{equation*}
I_\om(u) \geq C ||u||^2_{\Htau} - ||u||^{2\mu+2}_{2\mu+2} \geq C ||u||^{2}_{2\mu+2} - ||u||^{2\mu+2}_{2\mu+2},
\end{equation*}
for every $u \in \Htau$ from Sobolev inequality (\ref{Sobolev}), and C is a positive constant. Thanks to Lemma \ref{lemma2}, $u$ can be chosen in the region $I_\om(u) \leq 0$, hence it results that either $u=0$ or $||u||_{2\mu+2} \geq C^{\frac{1}{2\mu}}>0$. But since we are looking for non-zero minimizers, it follows that $||u||_{2\mu+2}$>0 and therefore $d(\om)>0$.
\end{proof}
\end{lem}

\noindent Finally, let us consider the following action functional with no point interactions
\begin{equation}
\label{nointeract}
S_\om^{0}(u)  =  \frac{1}{2} ||u'||_2^2 -\frac{1}{2\mu+2} ||u||_{2\mu+2}^{2\mu +2}+\frac{\om}{2} ||u||_2^2
\end{equation}
and its corresponding Nehari's functional
\begin{equation}
\label{nointneh}
I_\om^{0}(u) =   ||u'||_2^2 - ||u||_{2\mu+2}^{2\mu +2}+\om ||u||_2^2,
\end{equation}
defined on the space $\Htau$. From Section 8.4 of \cite{anv}, we know that for any $\tau>0$ and $\om>0$ the minimizer of the functional $S_\om^{0}$ among the functions in $\Htau \backslash \{0\}$ such that $I_\om^{0}=0$ is given by the solution of the dipole interaction problem
\begin{equation}
\label{dipole}
\dip(x)= \left( \omega (\mu+1) \right)^{\frac{1}{2\mu}} \cosh^{-\frac{1}{\mu}} \left( \mu \sqrt{\omega}(x - \zeta_\pm) \right),
\end{equation}
where $\zeta_\pm$ are defined by 
\begin{equation*}
\tanh \left( \mu \sqrt{\om} \zeta_- \right) = \sqrt{\frac{1-\tau^{2\mu}}{1-\tau^{2\mu+4}}}
\end{equation*}
\noindent and
\begin{equation*}
\tanh \left( \mu \sqrt{\om} \zeta_+ \right) = \tau^2 \sqrt{\frac{1-\tau^{2\mu}}{1-\tau^{2\mu+4}}}.
\end{equation*}
Note that through the same argument used in Lemma \ref{lemma2}, the search for a non-zero minimizer for the functional $S_\om^{0}$ on the manifold $\{ u \in \Htau \, : \, I_\om^{0}(u)=0 \}$ turns out to be equivalent to look for a minimizer for the functional $\widetilde{S}$ on the manifold $\{ u \in \Htau \, : \, I_\om^{0}(u) \leq 0 \}$, in particular for any $u \in \Htau$ such that $I_\om^0(u) \leq 0$, it holds 
\begin{equation}
\label{aaa}
\widetilde{S}(\dip) \leq \widetilde{S}(u). 
\end{equation}
Let us introduce a lemma that will be used in the following and links the original problem to the one with no point interactions.

\begin{lem}
\label{comp}
Let $\om > \frac{v^2}{(\tau ^2 + 1)^2}$. Then, $d(\om)<\widetilde{S}(\dip)$.
\begin{proof}
The proof of this lemma follows immediately from Remark \ref{idea}, noting that 
\begin{equation*}
I_\om(\dip)=I_\om^0(\dip) - v|\dip(0-)|^2<0,
\end{equation*}
because of (\ref{dipole}), therefore $\dip$ is not a minimizer for $\widetilde{S}$ on $I_\om \leq 0$.
\end{proof}
\end{lem} 

\noindent Now we are able to demonstrate Theorem \ref{existence}.
\begin{proof}
Let us consider a minimizing sequence $u_n$ for the functional $\widetilde{S}$ such that $I_\om(u_n) \leq 0$ and prove that it is bounded in the $\Htau$ norm.\\
By definition, $\widetilde{S}(u_n)\rightarrow d(\om)$ for $n \rightarrow \infty$, hence the sequence $||u_n||^{2\mu+2}_{2\mu+2}$ is bounded by a positive constant $C$. \\
Since $I_\om(u_n) \leq 0$, it follows that 
\begin{equation*}
||u'_n||_2^2-v|u_n(0-)|^2+\om||u_n||^2_2-||u_n||^{2\mu+2}_{2\mu+2} \leq 0
\end{equation*}
and thanks to the boundedness of the $L^{2\mu+2}$-norm we get
\begin{equation*}
||u'_n||_2^2-v|u_n(0-)|^2+\om||u_n||^2_2 \leq ||u_n||^{2\mu+2}_{2\mu+2} \leq C.
\end{equation*}
On the other hand, by the proof of Lemma \ref{lemma1} we know that there exists $a>0$ such that
\begin{equation*}
||u'_n||_2^2-v|u_n(0-)|^2+\om||u_n||^2_2 \geq \left(\om - \frac{va}{\tau^2+1} \right) ||u_n||^2_2
\geq 0.
\end{equation*}
Hence, owing to \eqref{eq:ebounds} we conclude that:
\begin{equation*}
\label{L2bound}
||u_n||^2_2 \leq \left( \om - \frac{va}{\tau^2+1} \right)^{-1}C
\end{equation*}
and the boundedness of the $L^2$-norm of the minimizing sequence is proved. To show the boundedness of the $L^2$-norm of the sequence of the derivatives, we can proceed in a similar way. In particular:
\begin{align*}
 ||u'_n||_2^2 & \leq v|u_n(0-)|^2-\om||u_n||^2_2+||u_n||^{2\mu+2}_{2\mu+2} \\
                       & \leq v|u_n(0-)|^2+||u_n||^{2\mu+2}_{2\mu+2}\\
                       & \leq \frac{v}{\tau^2+1} \left( a||u_n||^2_2 +\frac{1}{a}||u'_n||^2_2 \right) + C,
\end{align*}
where for the last inequality we used Lemma \ref{lemma1} and the boundedness of the $L^{2\mu+2}$-norm. Hence, by \eqref{eq:ebounds}
\begin{equation*}
\left(1-\frac{v}{a(\tau^2+1)} \right) ||u'_n||_2^2 \leq \frac{av}{\tau^2+1}||u_n||^2_2+C.
\end{equation*}
This proves that the $L^2$-norm of the sequence $u'_n$ is bounded and by (\ref{Htaunorm}) we conclude that the sequence $u_n$ is bounded in the $\Htau$-norm. 
\newline By Banach-Alaoglu's theorem there exists a subsequence (that we will still call $u_n$) that is weakly convergent in $\Htau$. We name $u$ its weak limit and prove that $u \neq 0$ and $I_\om(u)\leq 0$. Before showing that $u$ is non-vanishing, it is useful to prove that
\begin{equation}
\label{neharilimit}
\lim_{n \rightarrow \infty} I_\om(u_n)=0.
\end{equation}
This is proved by contradiction. Indeed, if we suppose that $\liminf I_\om(u_n) < 0$, then there would exist a subsequence denoted by $u_n$ again and we could define a sequence $v_n:=\beta_n u_n$, with 
\begin{equation*}
\beta_n := \left( \frac{||u'_n||^2_2-v|u_n(0-)|^2+\om ||u_n||^2_2}{||u_n||^{2\mu+2}_{2\mu+2}} \right)^ {\frac{1}{2\mu}}
\end{equation*}
and $\liminf \beta_n <1$. Hence, we would get that 
\begin{equation*}
\liminf \widetilde{S}(v_n) = \liminf {\beta_n}^{2\mu+2}\widetilde{S}(u_n) < \liminf \widetilde{S}(u_n),
\end{equation*} contradicting the hypothesis that $u_n$ is a minimizing sequence. Therefore $\liminf I_\om(u_n) \geq 0$, but since $\limsup I_\om(u_n) \leq 0$, it must be $\lim I_\om(u_n)=0$.\\ Finally, to prove that $u \neq 0$, we proceed again by contradiction, assuming that $u = 0$ and in particular $u(0+)=u(0-)=0$. We can define a sequence $h_n:=\rho_n u_n$ where
\begin{equation}
\rho_n := \left( \frac{||u'_n||^2_2+\om ||u_n||^2_2}{||u_n||^{2\mu+2}_{2\mu+2}} \right)^ {\frac{1}{2\mu}}.
\end{equation}
Because of the estimate $|u_n(0 \pm)-u(0 \pm)|\leq ||u_n-u||_{\Htau}$, it follows that $\lim u_n(0 \pm)=u(0 \pm)=0$ and thanks to (\ref{neharilimit}), we obtain
\begin{equation}
\lim \rho_n= \lim_{n \rightarrow \infty} \left(1+ \frac{I_\om(u_n)+v|u_n(0-)|^2}{||u_n||^{2\mu+2}_{2\mu+2}} \right)^{\frac{1}{2\mu}}=1.
\end{equation}
Therefore, it follows that $\lim \widetilde{S}(h_n)=\lim {\rho_n}^{2\mu+2}\widetilde{S}(u_n)=d(\om)$. \\On the other hand we observe that
\begin{equation*}
I^0_\om(h_n)= I^0_\om(\rho_n u_n)= \rho_n^2(||u'_n||^2_2+\om||u_n||^2_2-\rho_n^{2\mu}||u_n||^{2\mu+2}_{2\mu+2})=0.
\end{equation*}
By (\ref{aaa}) we can conclude that $d(\om) \geq \widetilde{S}(\dip)$, but by Lemma \ref{comp} we know that $d(\om)<\widetilde{S}(\dip)$. Thus, the assumption $u=0$ cannot hold.  \\\\
It remains to prove that $u$ belongs to the right manifold and in particular that $I_\om(u) \leq 0$. For this purpose we exploit Brezis-Lieb's lemma \cite{bl}, that establishes that: if $u_n \rightarrow u$ pointwise and $||u_n||_p$ is uniformly bounded, then 
\begin{equation}
\label{brezis-lieb}
||u_n||^p_p-||u_n-u||^p_p-||u||^p_p \rightarrow 0, \quad \forall 1<p<\infty.
\end{equation}
\noindent Then, owing to that result
\begin{equation}
\label{Sineq}
\widetilde{S}(u_n)-\widetilde{S}(u_n-u)-\widetilde{S}(u) \rightarrow 0,
\end{equation}
whereas by the weak convergence of $u_n$ in $\Htau$, it follows that
\begin{equation}
\label{nehariineq}
I_\om(u_n)-I_\om(u_n-u)-I_\om(u) \rightarrow 0.
\end{equation}
To show that $I_\om(u)\leq 0$ we proceed by contradiction assuming that $I_\om(u)>0$. Hence, from (\ref{nehariineq}) it follows that
\begin{equation*}
\lim I_\om(u_n-u) = \lim I_\om(u_n) - I_\om(u)=- I_\om(u)<0
\end{equation*}
thanks to (\ref{neharilimit}). This means that there exists a $\bar{n}$ such that for any $n > \bar{n}$, $I_\om(u_n-u)<0$ holds and therefore
\begin{equation}
\label{uno}
d(\om) < \widetilde{S}(u_n - u), \quad \forall n > \bar{n},
\end{equation} 
thanks to Remark \ref{idea}.\\
On the other hand, by (\ref{Sineq}) we get
\begin{equation}
\label{due}
\lim_{n \rightarrow \infty} \widetilde{S}(u_n-u) = \lim_{n \rightarrow \infty} \widetilde{S}(u_n) - \widetilde{S}(u)=d(\om)- \widetilde{S}(u)<d(\om)
\end{equation}
due to the fact that $\widetilde{S}(u)>0$, since $u\neq 0$.\\ Finally we note that (\ref{uno}) and (\ref{due}) are in contradiction, hence the hypothesis $I_\om(u)>0$ cannot hold.\\
By definition, one has that $d(\om) \leq \widetilde{S}(u)$, but on the other hand it holds
\begin{equation*}
\widetilde{S}(u)=\frac{\mu}{2(\mu+1)}||u||^{2\mu+2}_{2\mu+2} \leq \lim_{n \rightarrow \infty} \frac{\mu}{2(\mu+1)}||u_n||^{2\mu+2}_{2\mu+2}=d(\om),
\end{equation*}
because $u_n \rightarrow u$ weakly in $L^{2\mu+2}$. Hence, $u$ is the suitable minimizer and
\begin{equation}
\label{final}
\widetilde{S}(u)=d(\om).
\end{equation}
\end{proof}
\noindent We end this section presenting a stronger result about the convergence of a minimizing sequence in $\Htau$. In particular:
\begin{cor}
Every minimizing sequence converges strongly in $\Htau$.
\begin{proof}
Let us denote by $u_n$ a minimizing sequence. From (\ref{due}) and (\ref{final}) it follows that $u_n \rightarrow u$ strongly in $L^{2\mu+2}$. 
Moreover, thanks to (\ref{neharilimit}) and Remark \ref{idea}, one has
\begin{align*}
||u'_n||^2_2 + \om ||u_n||^2_2 & = I_{\om}(u_n) + ||u_n||^{2\mu+2}_{2\mu+2}+v|u_n(0-)|^2 \\
                                                 & \rightarrow ||u||^{2\mu+2}_{2\mu+2}+v|u(0-)|^2 \\
                                                 & =||u'||^2_2+\om ||u||^2_2.
\end{align*}
Thanks to (\ref{cane}), this implies strong convengence in $\Htau$ and complete the proof.
\end{proof}
\end{cor}

 \section{Ground States}
 \label{sec:ground}

In order to identify the ground state of $S_\om$, this section is devoted to study the stationary states of the constrained action functional and in particular to introduce the F\"{u}l\"{o}p-Tsutsui conditions presented in the Introduction. Then, we will detect the ground state among all the stationary states of the constrained functional.

 \subsection{Stationary States}

In the first part of this section we present some results about the stationary states of the functional $S_\om$; in particular we prove that they solve the stationary nonlinear Schr\"{o}dinger equation on each of the two halflines and own a discontinuity at the origin under some specific conditions, the so-called F\"{u}l\"{o}p-Tsutsui conditions.
\begin{prop}
A stationary state for the action functional $S_\om$ constrained on the Nehari manifold solves
\begin{equation}
\label{condition1}
\Bigg \{ \begin{array}{lr}
-u''-|u|^{2\mu}u+\omega u =0, \quad x\neq0, \quad u\in H^2(\mathbb{R}\backslash \{0\})\\
u(0+) =\tau u(0-) \\
u'(0-)-\tau u'(0+) =v u(0-)  \\
\end{array}
\end{equation}
\end{prop}

\begin{proof}
Let $u$ be a stationary state for the functional $S_\om$ constrained on the Nehari manifold, then there exists a Lagrange's multiplier $\nu \in \mathbb{R}$ such that $S'_{\om}(u)=\nu I'_{\om}(u)$ and $\langle S'_{\om}(u),u \rangle= \nu \langle I'_{\om}(u),u \rangle$.\\ On the other hand, by direct computation and stationarity, one obtains 
\begin{align*}
    \langle S'_{\om}(u),u \rangle & = I_{\om}(u) =0,\\
    \langle I'_{\om}(u),u \rangle  & = -2\mu||u||^{2\mu+2}_{2\mu+2}.
\end{align*} Hence, $\nu=0$ and the Euler-Lagrange equation becomes $S'_\om(u)=0$.\\
For any $\eta \in \Htau$ it follows that
\begin{align*}
\langle S'_\om(u), \eta \rangle =
\intRneg u'\eta' \dx & + \intRpos u'\eta' \dx + \\ & - \int_{-\infty}^{+\infty} (|u|^{2\mu+1} - \omega u) \eta \dx - v u(0-)\eta(0-) = 0.
\end{align*}
If we pick one of the two halflines and consider $\eta \in C_c ^{\infty}(\rr_+)$ or $\eta \in C_c ^{\infty}(\rr_-)$, the term $v u(0-)\eta(0-)$ vanishes and the equation $u''+|u|^{2\mu}u=\omega u$ holds on the halfline.\\ In order to verify the conditions at the origin, we proceed integrating by parts the l.h.s of the equation; it follows that, for any $\eta \in \Htau$, it holds:
\begin{align*}
u'\eta \bigg|_{-\infty}^0 + u'\eta \bigg |_0^{+\infty}  & -  \intRneg (u''+|u|^{2\mu+1} - \omega u) \eta \dx + \\ & -  \intRpos (u''+|u|^{2\mu+1} - \omega u) \eta \dx - v u(0-)\eta(0-) = 0.
\end{align*}

\noindent Hence, 
\begin{equation*}
u'(0-)\eta(0-)-u'(0+)\eta(0+) = v u(0-)\eta(0-)
\end{equation*}
\noindent and finally
\begin{equation*}
    u'(0-)-\tau u'(0+) = vu(0-),
\end{equation*}
concluding the proof.
\end{proof}

\noindent In the following result, we show that there exists a threshold such that there are no solutions if $\om$ is below that value. On the other hand, when $\om$ is above the threshold, there exist one or two solutions whose profile is given by pieces of the soliton  
\begin{equation}
\label{soliton}
\phi_{\om}(x)=\left( \omega (\mu+1) \right)^{\frac{1}{2\mu}} \cosh^{-\frac{1}{\mu}} \left( \mu \sqrt{\omega} x \right),
\end{equation}

\noindent one on each halfline and they match at the origin through the F\"{u}l\"{o}p-Tsutsui conditions (\ref{condition1}).

\begin{thm}
\label{stationary}
For $\om \leq \frac{v^2}{(\tau^2+1)^2}$ the system (\ref{condition1}) has no solutions. For every $\om \in \left( \frac{v^2}{(\tau^2+1)^2}, \frac{v^2}{(\tau^2-1)^2} \right]$ there exists a unique solution. Finally, for $\om > \frac{v^2}{(\tau^2-1)^2}$ a new branch of solutions arises separately from the previous one and there are two solutions (see Figure \ref{biforcazione}).\\
All solutions have the form
\begin{equation}
\label{sol}
u_{\om}(x) = 
\bigg \{ \begin{array}{rl}
\phi_{\om}(x+x_-), & x \in \mathbb{R}_{-} \\
\phi_{\om}(x+x_+), & x \in \mathbb{R}_{+} \\
\end{array}
\end{equation}
where  $\phi_\om$ was defined in (\ref{soliton}) and $x_-, x_+ \in \mathbb{R}$ are given by the solutions of the system 
\begin{equation}
\label{condition3}
\Bigg \{ \begin{array}{rl}
T_+=\frac{1}{\tau^2} \left( T_- + \frac{v}{\sqrt{\omega}} \right) \\
\frac{{T_-}^2}{1-\frac{1}{\tau^{2\mu}}}-\frac{{T_+}^2}{\tau^{2\mu}-1}=1 \\
\end{array}
\end{equation} 
in the unknowns $T_\pm=T_\pm(\om)=\tanh(\mu \sqrt{\om} x_\pm)$.
\end{thm}

\begin{proof}
By standard results \cite{anv}, it is known that the only solution of the equation $-u''+|u|^{2\mu}u-\omega u=0$ on each halfline is given by $\phi_\om(x+\bar{x})$, where $\bar{x}$ is a suitable real number and $\phi_\om$ was defined in (\ref{soliton}). Hence, on the real line $\rr$, the solution is given by (\ref{sol}). Therefore, in order to study the existence of the solutions of the system (\ref{condition3}), we need to check for which $x_\pm \in \rr$ the F\"{u}l\"{o}p-Tsutsui conditions are satisfied.\\
From the discontinuity condition $u(0+)=\tau u(0-)$ and thanks to (\ref{sol}) and (\ref{soliton}), it follows that
\begin{align*}
\phi_\om(x_+) & =\tau \phi_\om(x_-) \\
\cosh^{-\frac{1}{\mu}} \left( \mu \sqrt{\omega} x_+ \right) & =\tau \cosh^{-\frac{1}{\mu}} \left( \mu \sqrt{\omega} x_- \right) \\
(1-T_+^2)^{\frac{1}{2\mu}} & = \tau (1-T_-^2)^{\frac{1}{2\mu}} \\
(1-T_+^2) & = \tau^{2\mu} (1-T_-^2)
\end{align*}
where we used the fact that $\cosh^{-2}(x)=1-\tanh^{2} (x)$ and $T_\pm := \tanh \left( \mu \sqrt{\omega} x_\pm \right)$.\\
On the other hand, from $u'(0-)-\tau u'(0+)=v u(0-)$ and proceeding similarly, we get $T_+= \frac{1}{\tau ^2}\left(T_- + \frac{v}{\sqrt{\omega}} \right)$. \\
In this way the two conditions at the origin can be rewritten in the following system
\begin{equation*}
\Bigg \{ \begin{array}{rl}
T_+=\frac{1}{\tau^2} \bigg( T_- + \frac{v}{\sqrt{\omega}} \bigg) \\
\frac{{T_-}^2}{1-\frac{1}{\tau^{2\mu}}}-\frac{{T_+}^2}{\tau^{2\mu}-1}=1
\end{array}
\end{equation*}
So the proof is complete.
\end{proof}

\begin{rem}
Note that it is not restrictive to suppose that $\tau > 0$.
\end{rem}

\begin{figure}[H]
\centering
\includegraphics[width=0.55\columnwidth]{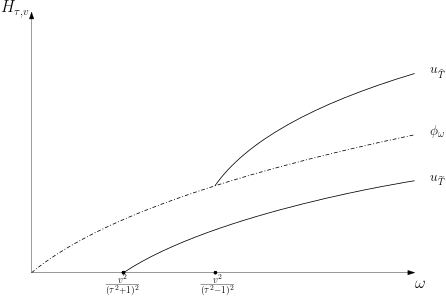}
\captionof{figure}{Qualitative graph of bifurcation for the stationary states depending on $\om$. Note that the dotted-dashed line refers to the soliton $\phi_\om \notin H_{\tau,v}$.}
\label{biforcazione}
\end{figure}

\subsection*{Solutions to the system (\ref{condition3})} 
System (\ref{condition3}) has an easy geometric representation, as shown in Figure \ref{iperbole}. 
Indeed, one can observe that the first equation of (\ref{condition3}) describes a line that approaches the origin for increasing $\omega$, but never reaches it because $v \neq 0$. On the other hand, the second equation represents a hyperbola that does not depend on $\om$ and crosses the vertices of the unitary square. The intersections between the line and the hyperbola give us the solutions to the system.\\ Moreover, for $\tau=\bar{\tau}$ and $\tau=\frac{1}{\bar{\tau}}$, there is a symmetry respect to the line $y=-x$ between the two hyperbola, whereas this symmetry is reached by the line only in the limit $\om \rightarrow \infty$.

\begin{figure}[H]
\centering
\includegraphics[width=0.6\columnwidth]{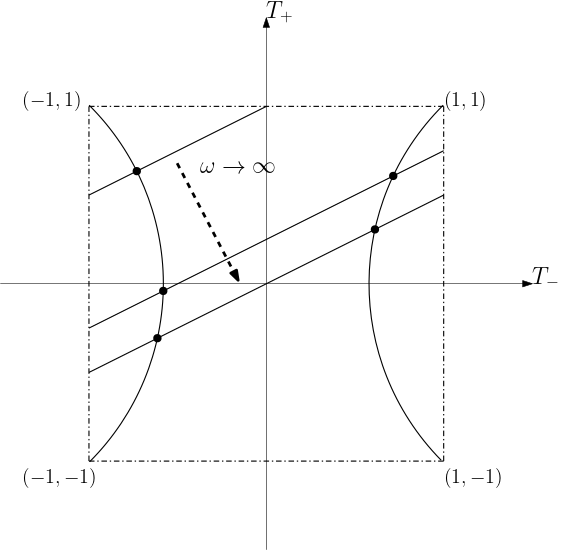}
\captionof{figure}{Geometric representation of the system (\ref{condition3}) for $\tau > 1$, where the dots represent the solutions to the system for $\om \rightarrow \infty$.}
\label{iperbole}
\end{figure}

\noindent By direct computation we obtain two couples of solutions, $(\widetilde{T}_-,\widetilde{T}_+)$ and $(\widehat{T}_- ,\widehat{T}_+)$, where:
\begin{align*}
\widetilde{T}_- = \widetilde{T}_-(\om) & =\tanh(\mu \sqrt{\omega} \widetilde{x}_-)= \\ 
                               & =\frac{1}{\tau^{2\mu+4}-1} \bigg(\frac{v}{\sqrt{\omega}}-\tau^2 \sqrt{\frac{v^2}{\omega}\tau^{2\mu}+(\tau^{2\mu+4}-1)(\tau^{2\mu}-1)} \bigg), \\
\widetilde{T}_+=\widetilde{T}_+(\om) & =\tanh(\mu \sqrt{\omega} \widetilde{x}_+)= \\ 
                             & \frac{1}{\tau^{2\mu+4}-1} \bigg(\tau^{2\mu+2}\frac{v}{\sqrt{\omega}}-\sqrt{\frac{v^2}{\omega}\tau^{2\mu}+(\tau^{2\mu+4}-1)(\tau^{2\mu}-1)} \bigg)
\end{align*}
and
\begin{align*}
\widehat{T}_-=\widehat{T}_-(\om) & =\tanh(\mu \sqrt{\omega} x_-)=\\ 
             & \frac{1}{\tau^{2\mu+4}-1} \bigg( \frac{v}{\sqrt{\omega}}+\tau^2 \sqrt{\frac{v^2}{\omega}\tau^{2\mu}+(\tau^{2\mu+4}-1)(\tau^{2\mu}-1)} \bigg), \\
\widehat{T}_+=\widehat{T}_+(\om) & =\tanh(\mu \sqrt{\omega} x_+)=\\
             & \frac{1}{\tau^{2\mu+4}-1} \bigg( \tau^{2\mu+2}\frac{v}{\sqrt{\omega}}+\sqrt{\frac{v^2}{\omega}\tau^{2\mu}+(\tau^{2\mu+4}-1)(\tau^{2\mu}-1)}  \bigg).
\end{align*}
\noindent Let us note that, since $\widehat{T}_\pm$ and $\widetilde{T}_\pm$ are defined as hyperbolic tangents, the solutions of (\ref{condition3}) must belong to the open unitary square. This is the reason why there are no admissible solutions for $\om \leq \frac{v^2}{(\tau^2+1)^2}$, there is a unique solution for any $\om \in \left( \frac{v^2}{(\tau^2+1)^2}, \frac{v^2}{(\tau^2-1)^2} \right]$ and there are two for  $\om > \frac{v^2}{(\tau^2-1)^2}$. \\\\ In order to identify which one between the two couples is the unique solution for any $\om \in \left( \frac{v^2}{(\tau^2+1)^2}, \frac{v^2}{(\tau^2-1)^2} \right]$, observe that neither $\widehat{T}_-$, nor $\widehat{T}_+$ does not change sign depending on $\om$: it is always positive if $\tau>1$ or always negative for $\tau<1$. \\
Since the first solution appears in the second quadrant, where $T_-$ is negative and $T_+$ is positive, we conclude that for $\om \in \left( \frac{v^2}{(\tau^2+1)^2}, \frac{v^2}{(\tau^2-1)^2} \right]$ the unique solution must be given by $(\widetilde{T}_-,\widetilde{T}_+)$.       \\
Finally, by the equivalence $\arctanh(x)= \frac{1}{2}\ln\left(\frac{1+x}{1-x} \right)$ we get the following identities:
\begin{align*}
\widetilde{x}_- & =\frac{1}{2\mu \sqrt{\om}} \ln \left( \frac{1-\tau^{2\mu+4}-\frac{v}{\sqrt{\om}}+\tau^2 \sqrt{\frac{v^2}{\omega}\tau^{2\mu}+(\tau^{2\mu+4}-1)(\tau^{2\mu}-1)}}{1-\tau^{2\mu+4}+\frac{v}{\sqrt{\om}}-\tau^2 \sqrt{\frac{v^2}{\omega}\tau^{2\mu}+(\tau^{2\mu+4}-1)(\tau^{2\mu}-1)}} \right), \\
\widetilde{x}_+ & =\frac{1}{2\mu \sqrt{\om}} \ln \left( \frac{1-\tau^{2\mu+4}-\tau^{2\mu+2} \frac{v}{\sqrt{\om}}+\sqrt{\frac{v^2}{\omega}\tau^{2\mu}+(\tau^{2\mu+4}-1)(\tau^{2\mu}-1)}}{1-\tau^{2\mu+4}+\tau^{2\mu+2}\frac{v}{\sqrt{\om}}- \sqrt{\frac{v^2}{\omega}\tau^{2\mu}+(\tau^{2\mu+4}-1)(\tau^{2\mu}-1)}} \right)
\end{align*}
and
\begin{align*}
\widehat{x}_- & =\frac{1}{2\mu \sqrt{\om}} \ln \left( \frac{1-\tau^{2\mu+4}-\frac{v}{\sqrt{\om}}-\tau^2 \sqrt{\frac{v^2}{\omega}\tau^{2\mu}+(\tau^{2\mu+4}-1)(\tau^{2\mu}-1)}}{1-\tau^{2\mu+4}+\frac{v}{\sqrt{\om}}+\tau^2 \sqrt{\frac{v^2}{\omega}\tau^{2\mu}+(\tau^{2\mu+4}-1)(\tau^{2\mu}-1)}} \right), \\
\widehat{x}_+ & =\frac{1}{2\mu \sqrt{\om}} \ln \left( \frac{1-\tau^{2\mu+4}-\tau^{2\mu+2} \frac{v}{\sqrt{\om}}-\sqrt{\frac{v^2}{\omega}\tau^{2\mu}+(\tau^{2\mu+4}-1)(\tau^{2\mu}-1)}}{1-\tau^{2\mu+4}+\tau^{2\mu+2}\frac{v}{\sqrt{\om}}+ \sqrt{\frac{v^2}{\omega}\tau^{2\mu}+(\tau^{2\mu+4}-1)(\tau^{2\mu}-1)}} \right).
\end{align*}

\noindent In the following we will refer to the stationary states of $S_\om$ as $u_{\widetilde{T}}$ and $u_{\widehat{T}}$.

\begin{figure}[H]
\subfloat{ \begin{tikzpicture}[xscale= 0.5,yscale=1.5]
    \draw[step=2,thin] (-7,0) -- (7,0);
\node at (0,0) [nodo] {};
\draw[dashed,thin] (-8.2,0)--(-7.2,0)  (7.2,0)--(8.2,0); 
\draw[color=black,domain=-7:0] plot(\x, {2 / cosh(\x - 2.5)  });
\draw[color=black,domain=0:7] plot(\x, {2 / cosh(\x + 2) });
\end{tikzpicture}}
\subfloat{\includegraphics[width=2in]{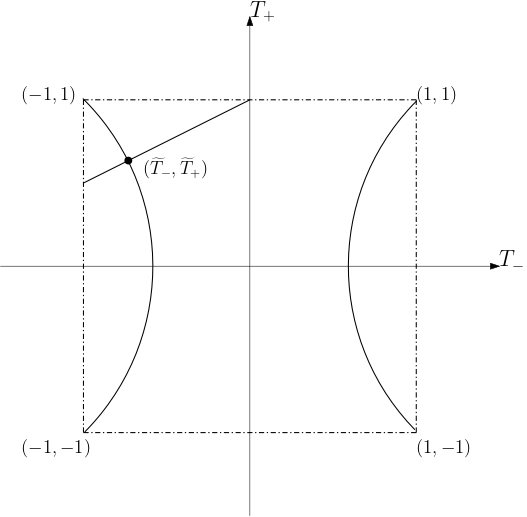}}
\end{figure}
\begin{figure}[H]
\subfloat{\begin{tikzpicture}[xscale= 0.5,yscale=1.5]
    \draw[step=2,thin] (-7,0) -- (7,0);
\node at (0,0) [nodo] {};
\draw[dashed,thin] (-8.2,0)--(-7.2,0)  (7.2,0)--(8.2,0); 
\draw[color=black,domain=-7:0] plot(\x, {2 / cosh(\x - 2.5)  });
\draw[color=black,domain=0:7] plot(\x, {2 / cosh(\x ) });
\end{tikzpicture}}
\subfloat{\includegraphics[width=2in]{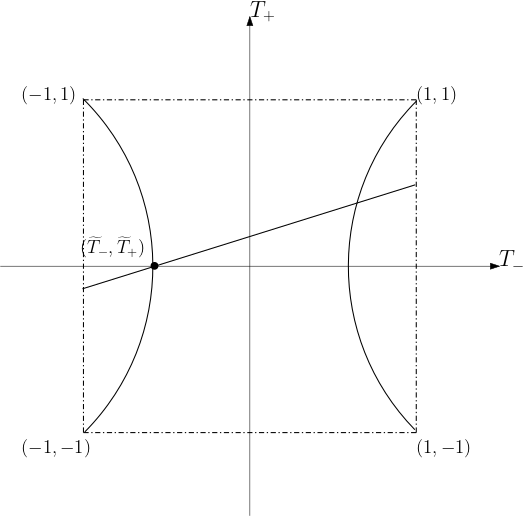}}
\end{figure}
\begin{figure}[H]
\subfloat{\begin{tikzpicture}[xscale= 0.5,yscale=1.5]
    \draw[step=2,thin] (-7,0) -- (7,0);
\node at (0,0) [nodo] {};
\draw[dashed,thin] (-8.2,0)--(-7.2,0)  (7.2,0)--(8.2,0); 
\draw[color=black,domain=-7:0] plot(\x, {2 / cosh(\x - 2.5)  });
\draw[color=black,domain=0:7] plot(\x, {2 / cosh(\x - 1.7) });
\end{tikzpicture}}
\subfloat{\includegraphics[width=2in]{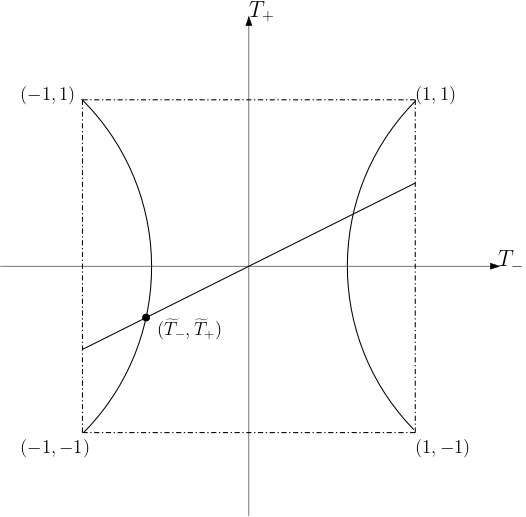}}
\end{figure}
\captionof{figure}{A sketch of the stationary state $u_{\widetilde{T}}(x)=\chi_- \phi_{\om}(x+\widetilde{x}_-)+\chi_+ \phi_{\om}(x+\widetilde{x}_+)$ corresponding to the solution $(\widetilde{T}_-,\widetilde{T}_+)$ to the system \eqref{condition3} and $\widetilde{T}_\pm=\tanh(\mu \sqrt{\om} \widetilde{x}_\pm)$. It has always the profile of a tail of a soliton on the negative halfline, whereas on the positive halfline, depending on $\om$, it can be a tail, a half soliton or presents a bump.}
\label{st1}

\begin{figure}[H]
\subfloat{
\begin{tikzpicture}[xscale= 0.5,yscale=1.5]
\draw[step=2,thin] (-7,0) -- (7,0);
\node at (0,0) [nodo] {};
\draw[dashed,thin] (-8.2,0)--(-7.2,0)  (7.2,0)--(8.2,0); 
\draw[color=black,domain=-7:0] plot(\x, {2 / cosh(\x + 5)  });
\draw[color=black,domain=0:7] plot(\x, {2 / cosh(\x + 2.5) });
\end{tikzpicture}
}
\subfloat{\includegraphics[width=2in]{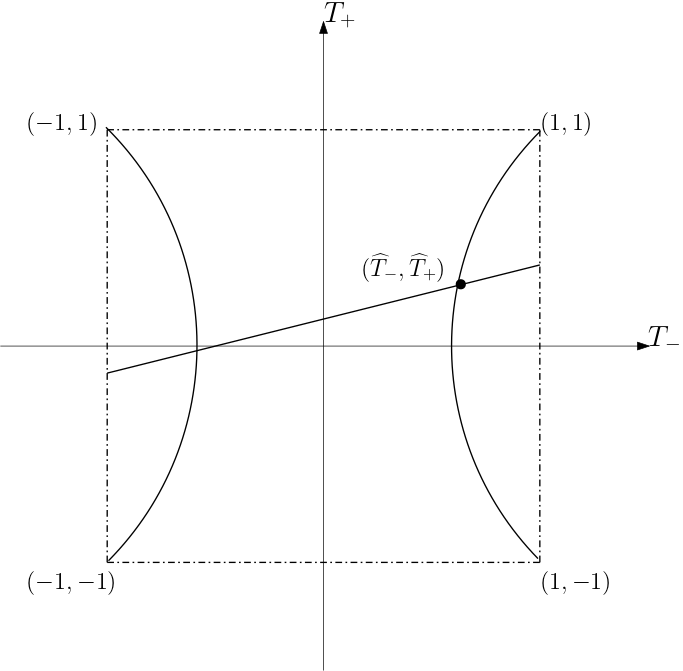}}
\end{figure}
\captionof{figure}{A sketch of the stationary state $u_{\widehat{T}}(x)=\chi_- \phi_{\om}(x+\widehat{x}_-)+\chi_+ \phi_{\om}(x+\widehat{x}_+)$ corresponding to the solution $(\widehat{T}_-,\widehat{T}_+)$ to the system \eqref{condition3} and $\widehat{T}_\pm=\tanh(\mu \sqrt{\om} \widehat{x}_\pm)$. This stationary state, regardless of $\om$, has always the profile of a tail of a soliton on the positive halfline and has a bump on the negative one.}
\label{st2}
\vspace{1cm}

 \subsection{Identification of the Ground State}
 
The main aim of this subsection is to identify the ground state among the stationary states of the action functional under the Nehari's constraint. But before going on with the search, we present some useful identities that hold for the stationary states presented previously.

\begin{prop}
\label{norme}
\label{ident} 
Let $u_{\widehat{T}}$ and $u_{\widetilde{T}}$ be the stationary states that solve the equation $u''+|u|^{2\mu}u=\omega u$ on each halfline and satisfy the F\"{u}l\"{o}p-Tsutsui conditions at the origin. Then the following identities hold:
\begin{align*}
\|u'_{\widehat{T}} \|^2_2  = \frac{(\mu+1)^{\frac{1}{\mu}}}{2}\om^{\frac{1}{\mu}+\frac{1}{2}} \bigg( & \intone (1-t^2)^{\frac{1}{\mu}} \dt - \int_{{\widehat{T}}_-(\om)}^{{\widehat{T}}_+(\om)} (1-t^2)^{\frac{1}{\mu}} \dt + \\ & + {\widehat{T}}_+(\om)(1-{\widehat{T}}_+^2(\om))^{\frac{1}{\mu}}-{\widehat{T}}_-(\om)(1-{\widehat{T}}_-^2(\om))^{\frac{1}{\mu}} \bigg),
\end{align*}
\begin{align*}
||u_{\widehat{T}}||^2_2 &= \frac{(\mu+1)^{\frac{1}{\mu}}}{\mu} \om^{\frac{1}{\mu}-\frac{1}{2}} \left( \int_{-1}^{1} (1-t^2)^{\frac{1}{\mu}-1}\dt - \int_{{\widehat{T}}_-(\om)}^{{\widehat{T}}_+(\om)} (1-t^2)^{\frac{1}{\mu}-1}\dt \right),\\
\|u_{\widehat{T}} \|^{2\mu+2}_{2\mu+2} & =\frac{(\mu+1)^{\frac{1}{\mu}+1}}{\mu} \om^{\frac{1}{\mu}+\frac{1}{2}} \left( \intone (1-t^2)^{\frac{1}{\mu}} \dt - \int_{{\widehat{T}}_-(\om)}^{{\widehat{T}}_+(\om)} (1-t^2)^{\frac{1}{\mu}} \dt \right),\\
|u_{\widehat{T}}(0-)|^2 & = (\mu+1)^{\frac{1}{\mu}} \om^{\frac{1}{\mu}} (1-{\widehat{T}}_-^2(\om))^{\frac{1}{\mu}}.
\end{align*}
Similarly, we get the same identities for $u_{\widetilde{T}}.$
\begin{proof} By direct computation,
\begin{align*}
\|u'_{\widehat{T}} \|^2_2  & = \intRneg |u'_{\widehat{T}} |^2 \dx + \intRpos |u'_{\widehat{T}} |^2 \dx  \\ 
               & = \intRneg |\phi_\om'(x+{\widehat{x}}_-)|^2 \dx + \intRpos |\phi_\om'(x+{\widehat{x}}_+)|^2 \dx  \\
               & = \int_{-\infty}^{\widehat{x}_-} |\phi_\om'(s)|^2 \ds + \int_{\widehat{x}_+}^{+\infty} |\phi_\om'(s)|^2 \ds  \\
               & = (\mu+1)^{\frac{1}{\mu}}\om^{\frac{1}{\mu}+1} \bigg(  \int_{-\infty}^{\widehat{x}_-} \cosh^{-{\frac{2}{\mu}}} \left( \mu \sqrt{\omega} s \right)  \tanh^{2} \left( \mu \sqrt{\omega} s\right) \ds + \\ 
               & + \int_{\widehat{x}_+}^{+\infty} \cosh^{-{\frac{2}{\mu}}} \left( \mu \sqrt{\omega} s \right)  \tanh^{2} \left( \mu \sqrt{\omega} s\right) \ds \bigg) \\
               & = \frac{(\mu+1)^{\frac{1}{\mu}}}{\mu} \om^{\frac{1}{\mu}+\frac{1}{2}} \bigg( \int_{-\infty}^{\mu \sqrt{\omega} {\widehat{x}}_-} \cosh^{-\frac{2}{\mu}}(x) \tanh^{2}(x)\dx + \\
               & + \int_{\mu \sqrt{\omega} {\widehat{x}}_+}^{+\infty} \cosh^{-\frac{2}{\mu}}(x) \tanh^{2}(x) \dx \bigg)  \\
               & = \frac{(\mu+1)^{\frac{1}{\mu}}}{\mu} \om^{\frac{1}{\mu}+\frac{1}{2}} \bigg( \intone (1-t^2)^{\frac{1}{\mu}-1}t^2 \dt - \int_{\widehat{T}_-}^{{\widehat{T}}_+} (1-t^2)^{\frac{1}{\mu}-1}t^2 \dt \bigg),
\end{align*}
where for the last equality we used the identity $\cosh^{-2}(x)=1+\tanh^2(x)$ and the change of variable $t=\tanh (x)$.\\
Integrating by parts one obtains the first identity of Proposition \ref{norme} and similarly the others.
\end{proof}
\end{prop}

\noindent Finally, recalling that by ground state we mean any global minimizer of the constrained action functional, we present the main theorem of the section.  

\begin{thm}
Let $\om > \frac{v^2}{(\tau^2+1)^2}$, then the ground state of the action functional $S_\om$ under the Nehari's constraint is $u_{\widetilde{T}}$. 

\begin{proof}
If $\om \in \left( \frac{v^2}{(\tau^2+1)^2},\frac{v^2}{(\tau^2-1)^2} \right]$, then $\widetilde{u}$ is the only stationary state existing for $S_\om$ so, thanks to Theorem \ref{existence} that guarantees the existence of a minimizer, it must be a ground state.\\ For $\om >\frac{v^2}{(\tau^2-1)^2}$, instead, there are two different stationary states, $u_{\widetilde{T}}$ and $u_{\widehat{T}}$, hence we need to compare $S_\om(u_{\widetilde{T}})$ and $S_\om(u_{\widehat{T}})$.\\ However, recalling that $S_\om(u)=\widetilde{S}(u)$ holds for any $u$ in the Nehari manifold, we reduce to compare $\widetilde{S}(u_{\widehat{T}})$ and $\widetilde{S}(u_{\widetilde{T}})$. \\By the explicit expressions of $\widetilde{T}_\pm$ and $\widehat{T}_\pm$, we can note that for $\tau >1$ the following inequalities hold 
\begin{align*}
    |{\widehat{T}}_-|&>|{\widehat{T}}_+|,\\ |\widetilde{T}_-|&>|\widetilde{T}_+|,\\
    \widetilde{T}_-  < \widetilde{T}_+ &< {\widehat{T}}_+ < {\widehat{T}}_-,
\end{align*} as one can immediately verify by Figure \ref{st1} and Figure \ref{st2}. Hence, $$\int_{\widetilde{T}_-} ^{\widetilde{T}_+}(1-t^2)^{\frac{1}{\mu}}\dt >0$$ and $$\int_{{\widehat{T}}_-} ^{{\widehat{T}}_+}(1-t^2)^{\frac{1}{\mu}}\dt <0.$$\\ On the other hand, for $\tau<1$ it follows that:
\begin{align*}
    |{\widehat{T}}_-|&<|{\widehat{T}}_+|,\\
    |\widetilde{T}_-| &<|\widetilde{T}_+|,\\
    {\widehat{T}}_+ < {\widehat{T}}_- & < \widetilde{T}_-  < \widetilde{T}_+,
\end{align*} but $$\int_{\widetilde{T}_-} ^{\widetilde{T}_+}(1-t^2)^{\frac{1}{\mu}}\dt >0$$ and $$\int_{{\widehat{T}}_-} ^{{\widehat{T}}_+}(1-t^2)^{\frac{1}{\mu}}\dt <0$$ still hold. \\ Hence, thanks to Proposition \ref{ident}, in both cases  $$\widetilde{S}(u_{\widetilde{T}})<\widetilde{S}(u_{\widehat{T}})$$ holds and we conclude the proof. 
\end{proof}
\end{thm}

\begin{rem} 
The previous result can be easily visualized thanks to the profiles of the stationary states in Figure \ref{st1} and Figure \ref{st2}. Indeed, note that by the F\"ul\"op-Tsutsui conditions imposed at the origin, for every admissible $\om$, the stationary state $u_{\widetilde{T}}$ is always smaller than a soliton, whereas $u_{\widehat{T}}$ is always larger.
\end{rem}

\section{Stability of the ground state}

In this section we present some results on the orbital stability of the ground state and to this aim we will rely on the well known Grillakis-Shatah-Strauss theory \cite{gss1,gss2}. However, before proceeding with the investigation, we remind what we mean by orbital stability.

\begin{defn}
A stationary state $U$ is called orbitally stable if for any $\epsilon >0$ there exists $\delta >0$ such that
\begin{equation*}
\inf_{\theta \in [0,2\pi)} ||\psi_0 - e^{i \theta}U||_{\Htau} \leq \delta \Rightarrow \sup_{t \geq 0} \inf_{\theta \in [0,2\pi)} ||\psi(t) - e^{i \theta}U||_{\Htau} \leq \epsilon,
\end{equation*}
where $\psi(t)$ is the solution to the problem
\begin{equation*}
\Bigg \{ \begin{array}{lr}
i \partial_t \psi(t) = H_{\tau,v} \psi(t) - |\psi(t)|^{2\mu}\psi(t),\\
\psi(0)=\psi_0.
\end{array}
\end{equation*}
\end{defn}
\noindent The well-posedness of such problem was studied in \cite{adno}, hence here we focus on the issue of stability.\\
Let us consider $w \in \Htau$ and write $w(x)=a(x)+ib(x)$, with $a$ and $b$ real functions and introduce the second variation of the action around a real function $\psi$:
\begin{align*}
\label{secondvar}
S''_\om(\psi+sw)_{s=0} = & ||a'||^2_2+\om||a||^2_2-v|a(0-)|^2-(2\mu+1) \int_{-\infty}^{+\infty} a(x)^2 |\psi(x)|^{2\mu}\dx + \\
                                      & + ||b'||^2_2+\om||b||^2_2-v|b(0-)|^2- \int_{-\infty}^{+\infty} b(x)^2 |\psi(x)|^{2\mu}\dx.
\end{align*} 
Defining the operators $L_1$ and $L_2$ with domains $D(L_1)=D(L_2)=\{ H^2(\mathbb{R} \backslash \{0\} ), u(0+)=\tau u(0-), u'(0-) -\tau u'(0+)= vu(0-)\}$, acting as follows:
\begin{align*}
L_1 a & = -a'' + \om a-(2\mu+1)|\psi|^{2\mu} a,\\
L_2 b & = -b'' + \om b-|\psi|^{2\mu} b,
\end{align*}
we can rewrite the second variation as
\begin{equation*}
S''_\om(\psi+sw)_{s=0} = (L_1 a,a) + (L_2 b,b).
\end{equation*}

\noindent From now on, we consider the second variation of the action around the ground state $u_{\widetilde{T}}$ and prove the following propositions concerning the operators $L_1$ and $L_2$. They provide the spectral information required by the Grillakis-Shatah-Strauss theory.

\begin{prop}
\label{GSS1}
The operator $L_2$ is such that $Ker(L_2)=Span\{u_{\widetilde{T}}\}$ and $L_2 \geq 0$.
\begin{proof}
The first statement follows by (\ref{condition1}). To prove the second part of the proposition we note that at any $x \neq 0$, for any $\phi \in D(L_2)$ it holds
\begin{equation*}
-\phi''+\om \phi - |u_{\widetilde{T}}|^{2\mu} \phi= - \frac{1}{u_{\widetilde{T}}}\frac{d}{dx}\left(u_{\widetilde{T}}^2 \frac{d}{dx} \left( \frac{\phi}{u_{\widetilde{T.}}} \right) \right),
\end{equation*}
since $u_{\widetilde{T}}$ is a stationary state and never vanishes.\\
Recalling that $\phi$ is a complex function, while $u_{\widetilde{T}}$ is real, it follows that
\begin{align*}
(L_2 \phi, \phi)= & \int_{-\infty}^0 u_{\widetilde{T}}^2 \left| \frac{d}{dx} \left( \frac{\phi}{u_{\widetilde{T}}} \right) \right| ^2 dx + \int^{+\infty}_0 u_{\widetilde{T}}^2 \left| \frac{d}{dx} \left( \frac{\phi}{u_{\widetilde{T}}} \right) \right| ^2 dx +\\
                           &+ \phi'(0+) \overline{\phi(0+)}- \frac{u'_{\widetilde{T}}(0+)}{u_{\widetilde{T}}(0+)}|\phi(0+)|^2 - \phi'(0-) \overline{\phi(0-)}+ \frac{u'_{\widetilde{T}}(0-)}{u_{\widetilde{T}}(0-)}|\phi(0-)|^2,
\end{align*} 
where the first two integral terms are positive. On the other hand, we note that:
\begin{align*}
\phi'(0+) \overline{\phi(0+)}- \phi'(0-) \overline{\phi(0-)} & = \phi'(0+) \tau \overline{\phi(0-)}-\phi'(0-)\overline{\phi(0-)} \\
                                                                       & = -|\phi(0-)|^2v.           
\end{align*}
Hence 
\begin{align*}
\phi'(0+) \overline{\phi(0+)} - & \frac{u'_{\widetilde{T}}(0+)}{u_{\widetilde{T}}(0+)}|\phi(0+)|^2 - \phi'(0-) \overline{\phi(0-)}+ \frac{u'_{\widetilde{T}}(0-)}{u_{\widetilde{T}}(0-)}|\phi(0-)|^2 =\\
& = -v|\phi(0-)|^2 + \frac{u_{\widetilde{T}}(0+) u'_{\widetilde{T}}(0-)|\phi(0-)|^2 - u_{\widetilde{T}}(0-) u'_{\widetilde{T}}(0+)|\phi(0+)|^2}{u_{\widetilde{T}}(0-)u_{\widetilde{T}}(0+)} 
\end{align*}
\begin{equation*}
 = \frac{-vu_{\widetilde{T}}(0-)u_{\widetilde{T}}(0+)|\phi(0-)|^2+u_{\widetilde{T}}(0+) u'_{\widetilde{T}}(0-)|\phi(0-)|^2 - u_{\widetilde{T}}(0-) u'_{\widetilde{T}}(0+)|\phi(0+)|^2}{u_{\widetilde{T}}(0-)u_{\widetilde{T}}(0+)}
\end{equation*}
because $-vu_{\widetilde{T}}(0-)=-u'_{\widetilde{T}}(0-)+\tau u'_{\widetilde{T}}(0+)$, it follows:
\begin{align*}
& = \frac{\tau  u'_{\widetilde{T}}(0+) u_{\widetilde{T}}(0+)|\phi(0-)|^2 - u_{\widetilde{T}}(0-) u'_{\widetilde{T}}(0+)|\phi(0+)|^2}{u_{\widetilde{T}}(0-)u_{\widetilde{T}}(0+)}\\
& = \frac{u'_{\widetilde{T}}(0+) \left( \tau^2 |\phi(0-)|^2 - |\phi(0+)|^2 \right)}{u_{\widetilde{T}}(0+)}=0
\end{align*}
and this concludes the proof.
\end{proof}
\end{prop}


\begin{prop}
\label{GSS2}
Let $\om > \frac{v^2}{(\tau^2+1)^2}$, then the operator $L_1$ has a trivial kernel and a single negative eigenvalue.
\begin{proof}
From Proposition \ref{GSS1}, we know that $L_2 u_{\widetilde{T}}=0$. As a consequence, it follows that
\begin{align*}
\frac{d}{dx} \left( -u''_{\widetilde{T}}+\om u'_{\widetilde{T}} -|u_{\widetilde{T}}|^{2\mu}u'_{\widetilde{T}}\right) & =0, \quad x \neq 0\\
 -u'''_{\widetilde{T}}+\om u''_{\widetilde{T}} -(2\mu+1)|u_{\widetilde{T}}|^{2\mu}u''_{\widetilde{T}} & = 0, \quad x \neq 0.
\end{align*}
However, $u'_{\widetilde{T}}$ does not satisfy the F\"{u}l\"{o}p-Tsutsui conditions at the origin, so it is not in the kernel of $L_1$.\\
As a matter of fact, if we consider the equation:
\begin{equation} \label{bla}
    -\zeta''+\om \zeta - \frac{\om (\mu + 1)(2\mu+1)}{\cosh^2(\mu \sqrt{\om} x)}\zeta=0, \quad x \neq 0
\end{equation} its solution is given by the derivative of the soliton \eqref{soliton} that, up to a factor, corresponds to:
\begin{equation*}
\zeta(x)=\frac{\sinh(\mu \sqrt{\om}x)}{\cosh^{1+\frac{1}{\mu}}(\mu \sqrt{\om}x)}.
\end{equation*} 
Moreover, let us note that there could not exist a non square-integrable solution $\eta \notin Span(\zeta)$ to \eqref{bla} such that $\int_0 ^{+\infty} |\eta(x)|^2 dx < \infty$. Indeed, in that case, by invariance under reflection the function $\eta(-x)$ would be an other solution to \eqref{bla} such that $\int_{-\infty} ^0 |\eta(x)|^2 dx < \infty$ and there would be three linearly independent solutions to \eqref{bla}, whereas they have to be two.\\
As a consequence, the equation
\begin{equation*}
    -\zeta''+\om \zeta - \frac{\om (\mu + 1)(2\mu+1)}{\cosh^2(\mu \sqrt{\om}(x+\chi_-(x) \widetilde{x}_- +\chi_+(x)\widetilde{x}_+)}\zeta=0, \quad x \neq 0
\end{equation*}
\noindent is solved by $\zeta_\beta(x)= \chi_- \zeta(x+\widetilde{x}_-)+ \beta \chi_+ \zeta(x+\widetilde{x}_+)$, with $\beta \in \mathbb{C}$ to be found. Imposing the F\"{u}l\"{o}p-Tsutsui conditions at the origin to $\zeta_\beta$, namely
\begin{equation*}
\label{systbeta}
\Bigg \{ \begin{array}{lr}
\beta \zeta(\widetilde{x}_+)= \tau \zeta(\widetilde{x}_-),\\
\zeta'(\widetilde{x}_-) - \tau \beta \zeta'(\widetilde{x}_+)= v \zeta(\widetilde{x}_-).
\end{array}
\end{equation*}
From the first equation we obtain
\begin{equation*}
\beta= \tau \frac{\sinh(\mu \sqrt{\om}\widetilde{x}_-) \cosh^{1+\frac{1}{\mu}}(\mu \sqrt{\om}\widetilde{x}_+)}{\sinh(\mu \sqrt{\om}\widetilde{x}_+) \cosh^{1+\frac{1}{\mu}}(\mu \sqrt{\om}\widetilde{x}_-)}.
\end{equation*}
Hence, from the second equation it follows that 
\begin{equation*}
\frac{\mu - \sinh^2(\mu \sqrt{\om}\widetilde{x}_-)}{\sinh(\mu \sqrt{\om}\widetilde{x}_-)\cosh(\mu \sqrt{\om}\widetilde{x}_-)} - \tau^2 \frac{\mu - \sinh^2(\mu \sqrt{\om}\widetilde{x}_+)}{\sinh(\mu \sqrt{\om}\widetilde{x}_+)\cosh(\mu \sqrt{\om}\widetilde{x}_+)}= \frac{v}{\sqrt{\om}}.
\end{equation*}
Recalling that $\cosh^{2}(x) - \sinh^2(x)=1$, we obtain 
\begin{align*}
& \frac{\mu-(\mu+1) \tanh^2(\mu \sqrt{\om}\widetilde{x}_-)}{\tanh(\mu \sqrt{\om}\widetilde{x}_-)} - \tau^2 \frac{\mu-(\mu+1) \tanh^2(\mu \sqrt{\om}\widetilde{x}_+)}{\tanh(\mu \sqrt{\om}\widetilde{x}_+)}= \frac{v}{\sqrt{\om}}.
\end{align*}
Recalling that $\widetilde{T}_{\pm}=\tanh({\mu \sqrt{\om}\widetilde{x}_{\pm}})$ and that the couple $(\widetilde{T}_-,\widetilde{T}_+)$ solves \eqref{condition3}, thanks to the first equation in the system, it follows that
\begin{equation*}
\frac{1-\widetilde{T}^2_-}{\widetilde{T}_-} - \tau^2 \frac{1-\widetilde{T}^2_+}{\widetilde{T}_+}=0.
\end{equation*}
Finally, using the second equation in (\ref{condition3}), one obtains that $$\widetilde{T}_-^2=\frac{\tau^{2\mu}-1}{\tau^{2\mu}(1-\tau^{2\mu+4})},$$ but this is impossible because the r.h.s. is negative. Hence, we conclude that the kernel of $L_1$ is trivial.\\
To prove the existence of a single negative eigenvalue for $L_1$, we first note that the number of negative eigenvalues is finite thanks to the the fast decay in $x$ and the boundedness of the last term in the l.h.s of \eqref{bla}. By Lemma \ref{below}, Proposition \ref{GSS1} and by the fact that the Nehari manifold has codimension one, we conclude that $L_1$ has at most one negative eigenvalue. On the other hand it holds
\begin{align*}
(L_1 u_{\widetilde{T}},u_{\widetilde{T}}) & = (L_2 u_{\widetilde{T}},u_{\widetilde{T}}) - 2\mu||u_{\widetilde{T}}||^{2\mu+2}_{2\mu+2} \\
                                                 & = - 2\mu||u_{\widetilde{T}}||^{2\mu+2}_{2\mu+2} <0.
\end{align*}
As a consequence $L_1$ has one negative eigenvalue.
\end{proof}
\end{prop}
\noindent In the remaining part of the section, we focus on the requirements regarding the $L^2$-norm of the ground state $u_{\widetilde{T}}$, in order to get orbital stability.
\begin{lem}
\label{decreasing}
\begin{equation*}
\varphi(\om)= \int_{\widetilde{T}_-(\om)}^{\widetilde{T}_+(\om)} (1-t^2)^{\frac{1}{\mu}-1}dt
\end{equation*}
 is a decreasing function of $\om$.
\begin{proof}
From the explicit form of $\widetilde{T}_-(\om)$ and $\widetilde{T}_+(\om)$ we obtain
\begin{align*}
\widetilde{T'}_-(\om) & = -\frac{v}{2(\tau^{2\mu+4}-1)}\left( \frac{1}{\om^{\frac{3}{2}}} - \frac{v \tau^{2\mu+2}}{\om^2 \sqrt{A(\om)}} \right),\\
\widetilde{T'}_+(\om) & = -\frac{v}{2(\tau^{2\mu+4}-1)}\left( \frac{\tau^{2\mu+2}}{\om^{\frac{3}{2}}} - \frac{v \tau^{2\mu}}{\om^2 \sqrt{A(\om)}} \right),
\end{align*}
where $A(\om)=\frac{v^2}{\omega}\tau^{2\mu}+(\tau^{2\mu+4}-1)(\tau^{2\mu}-1)$.\\
 By (\ref{condition3}) it follows that $1- \widetilde{T}^2_- (\om)=\frac{1- \widetilde{T}^2_+(\om)}{\tau^{2\mu}}$, hence
\begin{align*}
\varphi'(\om) & =\left( 1- \widetilde{T_+}^2(\om) \right)^{\frac{1}{\mu}-1} \widetilde{T'}_+(\om) - \left( 1- \widetilde{T_-}^2(\om) \right)^{\frac{1}{\mu}-1} \widetilde{T'}_-(\om)\\
                     & = \left( 1- \widetilde{T_+}^2(\om) \right)^{\frac{1}{\mu}-1} \left( \widetilde{T'}_+ - \frac{\widetilde{T'}_-}{\tau^{2-2\mu}}\right).
\end{align*}
Recalling that $\widetilde{T}_\pm(\om) \in (-1,1)$, the first term in the r.h.s is positive. On the other hand, by direct computation one obtains
\begin{equation*}
\widetilde{T'}_+ - \frac{\widetilde{T'}_-}{\tau^{2-2\mu}} = - \frac{v}{2(\tau^{2\mu+4}-1)} \left( \frac{\tau^{2\mu+2}}{\om^{\frac{3}{2}}} (\tau^4-1) + \frac{v\tau^{2\mu}}{\om^2 \sqrt{A(\om)}}(\tau^{2\mu}-1) \right) <0.
\end{equation*} 
\end{proof}
\end{lem}
As a consequence, it follows that:
\begin{prop}
\label{GSS3}
Let $\om > \frac{v^2}{(\tau ^2 + 1)^2}$ and $\mu \in (0, 2]$. Then $M(\om)=||u_{\widetilde{T}}||^2_2$ is an increasing function of $\om$. 
\begin{proof}
Thanks to Proposition \ref{ident} and Lemma \ref{decreasing}, we can observe that $$M(\om)= \frac{(\mu+1)^{\frac{1}{\mu}}}{\mu} \om^{\frac{1}{\mu}-\frac{1}{2}} \left( \int_{-1}^{1} (1-t^2)^{\frac{1}{\mu}-1}dt - \varphi(\om) \right).$$ \\It follows that 
\begin{equation*}
M'(\om)= \xi'(\om) \left( \int_{-1}^{1} (1-t^2)^{\frac{1}{\mu}-1}dt - \varphi(\om) \right) - \xi(\om) \varphi'(\om),
\end{equation*}
where 
\begin{align*}
\xi(\om) & =\frac{(\mu+1)^{\frac{1}{\mu}}}{\mu} \om^{\frac{1}{\mu}-\frac{1}{2}},\\
\xi'(\om) & = \frac{(\mu+1)^{\frac{1}{\mu}}}{\mu}\frac{2-\mu}{2\mu}\om^{\frac{1}{\mu}-\frac{3}{2}}.
\end{align*}
Since $ \int_{-1}^{1} (1-t^2)^{\frac{1}{\mu}-1}dt - \varphi(\om)>0$, by Lemma \ref{decreasing} we just need to study the sign of $\xi(\om)$ and $\xi'(\om)$.\\
For $\mu \in (0,2)$ we have that $\xi(\om)>0$ and in particular $\xi(\om)$ is a positive constant for $\mu=2$. On the other hand, for $\mu \in (0,2]$, it follows that $\xi'(\om) \geq 0$.\\
\end{proof}
\end{prop}
\noindent We conclude with the main theorem of the section that collects all the previous results.
\begin{thm}
Let $\om > \frac{v^2}{(\tau ^2 + 1)^2}$, then for $\mu \in (0,2]$ the ground state $u_{\widetilde{T}}$ is orbitally stable. 
\begin{proof}
The proof follows from Proposition \ref{GSS1}, \ref{GSS2} and \ref{GSS3}. 
\end{proof}
\end{thm}

\begin{rem}
Relying on numerical results (see Figure \ref{supercr}), we conjecture that for $\mu>2$, the ground state $u_{\widetilde{T}}$ is stable up to a critical value of $\om$ and then, it becomes unstable.  
\begin{figure}[H]
\centering
\includegraphics[width=0.5\columnwidth]{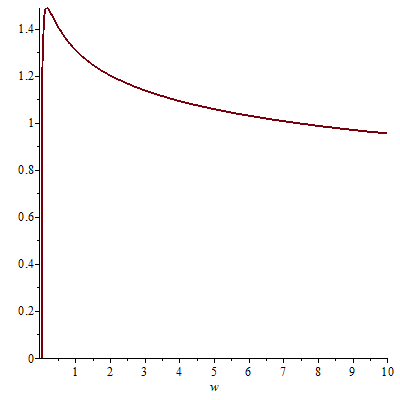}
\captionof{figure}{Graph of the mass of $u_{\widetilde{T}}$ depending on $\om$, for  $\mu=3$ and $v=1$, $\tau=2$.}
\label{supercr}
\end{figure}
\end{rem}


\begin{thebibliography}{99}

\bibitem{abd} Adami R., Boni F., Dovetta S., {\em Competing nonlinearities in NLS equations as source of threshold phenomena on star graphs}, arXiv:2009.06424 [math.AP] (2020).

\bibitem{abr} Adami R., Boni F., Ruighi A., {\em Non-Kirchhoff Vertices and NLS Ground States on graphs}, Mathematics (2020), 8, 617.

\bibitem{acfn2012} Adami R., Cacciapuoti C., Finco D., Noja D., {\em Stationary states of NLS on star graphs}, EPL (Europhysics Letters) 100 (1), 10003 (2012).

\bibitem{acfn2013} Adami R., Cacciapuoti C., Finco D., Noja D., {\em Constrained energy minimization and orbital stability for the NLS equation on a star graph}, Ann. I. H. Poincaré – AN (2013).

\bibitem{acfn2016} Adami R., Cacciapuoti C., Finco D., Noja D., {\em Stable standing waves for a NLS on star graphs as local minimizers of the constrained energy}, Journal of Differential Equations 260 (10), 7397--7415 (2016).

\bibitem{adr} Adami R., Dovetta S., Ruighi A., {\em Quantum graphs and dimensional crossover: the honeycomb}, Communications in Applied and Industrial Mathematics - ISSN 2038-0909. - 10:1, 109--122 (2019).

\bibitem{adst} Adami R., Dovetta S., Serra E., Tilli P., {\em Dimensional crossover with a continuum of critical exponents for NLS on doubly periodic metric graphs}, Analysis \& PDE, Vol. 12 (2019), No. 6, 1597–-1612.

\bibitem{adno} Adami R., Noja D., {\em Existence of dynamics for a 1D NLS equation perturbed with a generalized point defect}, Journal of Physics. A, Mathematical and Theoretical - ISSN 1751-8113. - 42(2009), pp. 495302-495320.

\bibitem{an} Adami R., Noja D., {\em Stability and simmetry breaking bifurcation for the ground states of a NLS equation with a $\delta'$ interaction}, Communications in Mathematical Physics 318 (1), 247--289 (2013).

\bibitem{anv} Adami R., Noja D., Visciglia N., {\em Constrained energy minimization and ground states for NLS with point defects}, Discrete Contin. Dyn. Syst. Ser. B., 18 (2013), 1155–1188.

\bibitem{ast1} Adami R., Serra E., Tilli P., {\em NLS ground states on graphs}, Calc. Var. and PDEs {\bf 54} (1), 743–761 (2015).

\bibitem{ast2} Adami R., Serra E., Tilli P., {\em Threshold phenomena and existence results for NLS ground states on graphs}, J. Funct. An. {\bf 271} (1), 201-223 (2016).

\bibitem{ast3} Adami R., Serra E., Tilli P., {\em Negative energy ground states for the L2-critical NLSE on metric graphs}, Commun. Math. Phys. {\bf 352}, no. 1, 387-406 (2017).

\bibitem{meh} Ali Mehmeti F., {\em Nonlinear waves in networks}, Akademie Verlag Berlin (1994).

\bibitem{bd} Boni F., Dovetta S., {\em Ground states for a doubly nonlinear Schr\"{o}dinger equation in dimension one}, arXiv:1907.07926 [math.AP] (2019).

\bibitem{berk} Berkolaiko G., Kuchment P., {\em Introduction to Quantum Graphs}, Mathematical Survey and Monographs, Applied Mathematics, American Mathematical Society (2013).

\bibitem{bmp} Berkolaiko G., Marzuola J.L., Pelinovsky D., {\em Edge-localized states on quantum graphs in the limit of large mass}. ArXiv:1910.03449 (2019).

\bibitem{bl} Brezis H., Lieb E.H., {\em A relation between pointwise convergence of functions and convergence of functional}, Proc. Amer. Math. Soc. 88 (3) (1983) 486–490.

\bibitem{cds} Cacciapuoti C., Dovetta S., Serra E., {\em Variational and stability properties of constant solutions to the NLS equation on compact metric graphs}, Milan Journal of Mathematics, 86 (2) (2018), 305-327.

\bibitem{cheon} Cheon T., Turek O., {\em Fulop–Tsutsui interactions on quantum graphs}, Phys. Lett. A (2010), 374, 4212–4221.

\bibitem{simon} Dovetta S., {\em Existence of infinitely many stationary solutions of the L2-subcritical and critical NLSE on compact metric graphs}, J. Differential Equations 264 (2018), no. 7, 4806-4821.

\bibitem{dov} Dovetta S., {\em Mass-constrained ground states of the stationary NLSE on periodic metric graphs}, Nonlinear Differ. Equ. Appl. {\bf 26} (2019), n. 30.

\bibitem{dst} Dovetta S., Serra E., Tilli P., {\em NLS ground states on metric trees: existence results and open questions}, J. London Math. Soc. (2020), article in press, published online. https://doi.org/10.1112/jlms.12361

 \bibitem{dt} Dovetta S., Tentarelli L., {\em $L^2$--critical NLS on noncompact metric graphs with localized nonlinearity: topological and metric features}, Calc Var. PDE {\bf 58} (3) (2019), n. 108.

\bibitem{fj} Fukuizumi R., Jeanjean L., {\em Stability of standing waves for a nonlinear Schr\"{o}dinger equation with a repulsive Dirac delta potential}, Disc. Cont. Dyn. Syst. (A), {\bf 21}, 129--144 (2008).

\bibitem{foo} Fukuizumi R., Otha M., Ozawa T., {\em Nonlinear Schr\"{o}dinger equation with a point defect}, Ann. IHP, Analyse non linéaire, {\bf 25}, 837--845 (2008).

\bibitem{good} Goodman R. H., {\em NLS Bifurcations on the bowtie combinatorial graph and the dumbbell metric graph}, Discr. Cont. Dyn. Systems - A, (2019), 39 (4) : 2203-2232.

\bibitem{gss1} Grillakis M., Shatah J., Strauss W., {\em Stability theory of solitary waves in the presence of simmetry - I}, J. Func. An., {\bf 74}, 160--197 (1987).

\bibitem{gss2} Grillakis M., Shatah J., Strauss W., {\em Stability theory of solitary waves in the presence of simmetry - I}, J. Func. An., {\bf 94}, 308-348 (1990).

\bibitem{kost} Kostrykin V., Schrader R., {\em Kirchhoff ’s rule for quantum wires}, J. Phys. A: Math. Gen. {\bf 32} (1999), 595—630.

\bibitem{marpel} Marzuola J. L., Pelinovsky D., {\em Ground state on the dumbbell graph}, Appl. Math. Res. Express 2016, no. 1 (2016), 98-145.

\bibitem{rued} Ruedenberg K., Scherr C. W., {\em Free-Electron Network Model for Conjugated Systems. I. Theory}, J. Chem. Phys. {\bf 21}, no. 9 (1953), 1565–1581.

\bibitem{serraten} Serra E., Tentarelli L., {\em Bound states of the NLS equation on metric graphs with localized nonlinearities}, J. Diff. Eq. {\bf 260} (2016), no. 7, 5627--5644.

\bibitem{serraten2} Serra E., Tentarelli L., {\em On the lack of bound states for certain NLS equations on metric graphs}, Nonlinear Anal. 145 (2016), 68--82.

\bibitem{shatah} Shatah J., {\em Stable standing waves of nonlinear Klein–Gordon equations}, Commun. Math. Phys. 91 (1983), 313–327.

\bibitem{t} Tentarelli L., {\em NLS ground states on metric graphs with localized nonlinearities}, J. Math. Anal. Appl. {\bf 433} (2016), no. 1, 291--304.
	
\end{thebibliography}
\end{document}